\documentclass[12pt,a4paper]{article}
\usepackage{fullpage}
\usepackage[affil-it]{authblk}
\usepackage{amsmath,amsfonts,amssymb,amsthm,amscd}
\newtheorem{Theorem}{Theorem}[section]

\newtheorem{Lemma}[Theorem]{Lemma}
\newtheorem{Corollary}[Theorem]{Corollary}

\newtheorem{Example}[Theorem]{Example}

\newcommand{\ZI}{\mathbb{Z}_{\mathbb{I}}}

\DeclareMathOperator{\lcm}{lcm}
\DeclareMathOperator{\ord}{ord\ }

\title{An explicit Baker type lower bound of exponential values}
\author{Anne-Maria Ernvall-Hyt\"onen\footnote{Department of Mathematics and Statistics, University of Helsinki}, Kalle Lepp\"al\"a\footnote{Department of Mathematical Sciences, University of Oulu} and Tapani Matala-aho\footnote{Department of Mathematical Sciences, University of Oulu} \footnote{The work of all authors was supported by the Academy of Finland, grant 138522, and the work of A.-M. E.-H. was also supported by the Academy of Finland grant 138337.}
}
\begin{document}
\maketitle
\date{}

\begin{abstract}
Let $\mathbb{I}$ denote an imaginary quadratic field or the field $\mathbb{Q}$
of rational numbers and $\mathbb{Z}_{\mathbb{I}}$ its ring of intergers.
We shall prove an explicit Baker type lower bound for
$\mathbb{Z}_{\mathbb{I}}$-linear form of the numbers
\begin{equation}\label{1}
1,\ e^{\alpha_1},...,\ e^{\alpha_m},\quad m\ge 2,
\end{equation}
where $\alpha_0=0$, $\alpha_1,...,\alpha_m$,
are $m+1$ different  numbers from the field $\mathbb{I}$.
Our work  gives gives some improvements to  the existing explicit versions of
of Baker's work about exponential values at rational points. In particilar,
dependences on $m$ are improved.

\end{abstract}

\section{Introduction}

Let $\mathbb{I}$ denote an imaginary quadratic field or the field $\mathbb{Q}$
of rational numbers and $\mathbb{Z}_{\mathbb{I}}$ its ring of intergers.
Our target is to prove an explicit Baker type lower bound for
$\mathbb{Z}_{\mathbb{I}}$-linear form of the numbers
\begin{equation}\label{1}
1,\ e^{\alpha_1},...,\ e^{\alpha_m},\quad m\ge 2,
\end{equation}
where $\alpha_0=0$, $\alpha_1,...,\alpha_m$,
are $m+1$ different  numbers from the field $\mathbb{I}$, i. e., we want to get a lower bound for the expression
\[
\left|\beta_1+\beta_1e^{\alpha_1}+\cdots +\beta_me^{\alpha_m}\right|,
\]
where $\beta_0,\beta_1,\dots ,\beta_m\in \mathbb{Z}_{\mathbb{I}}$. We will present the new explicit Baker type lower bound
\begin{equation*}
\left|\beta_0+\beta_1 e+\beta_1 e^{2}+...+\beta_me^{m}\right| >
\frac{1}{h^{1+\hat \epsilon(h)}},
\end{equation*}
valid for all
$$\overline{\beta}=(\beta_0,...,\beta_m)^T \in \mathbb{Z}_{\mathbb{I}}^{m}\setminus\{\overline{0}\},
\quad h_i= \max\{1,|\beta_i|\}, \quad  h=h_1\cdots h_m,$$
with
\begin{equation}\label{Aepsilon}
\hat \epsilon(h)=\frac{(4+7m)\sqrt{\log(m+1)}}{\sqrt{\log\log h}},
\end{equation}
\begin{equation*}
\log h \ge
m^2(41\log(m+1)+10)e^{m^2(81\log(m+1)+20)}.
\end{equation*}

With the assumption that $\gamma_0, \gamma_1,...,\gamma_m\in\mathbb{Q}^*$ are distinct,
Baker  \cite{BAKER} proved that there exist  positive constants $\delta_1,\delta_2$ and $\delta_3$ such that
\begin{equation}\label{Bakerresult}
\left|\beta_0 e^{\gamma_0}+...+\beta_m e^{\gamma_m}\right| >
\frac{\delta_1 M^{1-\delta(M)}}{\prod_{i=0}^{m} h_i},
\end{equation}
for all
$$\overline{\beta}=(\beta_0,...,\beta_m)^T \in \mathbb{Z}^{m}\setminus\{\overline{0}\},
\quad h_i= \max\{1,|\beta_i|\},$$
with
\begin{equation}\label{Bepsilon}
\delta(M)\le \frac{\delta_2}{\sqrt{\log\log M}},\quad
M=\underset{0\le i\le m}\max\{|\beta_i|\}\ge \delta_3>e.
\end{equation}
However, the constants $\delta_1,\delta_2,\delta_3$ in (\ref{Bakerresult}) are not explicitly given. Baker's work is based on the first kind Pad\'e-type approximations
for the vector $(e^{\alpha_1 t},...,e^{\alpha_m t})$, constructed by the Siegel's (Thue-Siegel's) lemma.

With similar methods, but using a more refined version of Siegel's lemma, together with careful analysis and estimates, and assuming that $\gamma_0, \gamma_1,...,\gamma_m\in\mathbb{Q}$ are distinct,
Mahler  \cite{MAHLER} made Baker's result completely explicit by proving that
\begin{equation}\label{Mahlerresult}
\left|\beta_0 e^{\gamma_0}+\beta_1 e^{\gamma_1}+...+\beta_m e^{\gamma_m}\right| >
\frac{M^{1-\delta(M)}}{\prod_{i=0}^{m} h_i},
\end{equation}
for all
$$\overline{\beta}=(\beta_0,...,\beta_m)^T \in \mathbb{Z}^{m+1}\setminus\{\overline{0}\}$$
with
\begin{equation}\label{Mepsilon}
\delta(M)\le \frac{12(m+1)^3\sqrt{\log B}}{\sqrt{\log\log M}},\quad
M=\underset{0\le i\le m}\max\{|\beta_i|\},
\end{equation}
\begin{equation*}
\log M \ge (16(m+1)^4\log B) e^{16(m+1)^4\log B},\quad
B=\underset{0\le i\le m}{\textrm{lcd}}\{\gamma_i\}(1+\underset{0\le i\le m}\max\{|\gamma_i|\}),
\end{equation*}
where lcd$\{\gamma_i\}$ denotes the least common denominator of the numbers $\gamma_i$.

Later, V\"a\"an\"anen and Zudilin \cite{VAAZUD} proved
Baker-type results for certain $q$-series solutions of $q$-difference equations.
V\"a\"an\"anen and Zudilin applied Pad\'e-type approximations of the second kind
with a  method similar to a one Baker \cite{BAKER} already used in earlier studies of hypergeometric series,
see e.g. \cite{ZUDILIN} and see \cite{FELDMAN} for an entire collection of relevant references.
Thereafter, Sankilampi \cite{SANKILAMPI} adapted the method of the Pad\'e-type approximations of the second kind
used in \cite{VAAZUD}, and  proved explicit Baker type results of exponential values comparable to Mahler's results.

Baker \cite{BAKER} and Mahler \cite{MAHLER} give their results
in the field of rational numbers while, Sankilampi \cite{SANKILAMPI} gives the results in  an arbitrary
imaginary quadratic field $\mathbb{I}$.
However, both Mahler and Sankilampi give their explicit results only in the field of rational numbers. Our results improve these. We will discuss and compare the results in Chapter \ref{Results}.

We start our considerations by introducing a refined version
of Siegel's lemma, Lemma \ref{Thuemax}, for an imaginary quadratic field
$\mathbb{Q}(\sqrt{-D})$, where $D\in \mathbb{Z}^+$ .
Our result improves the $D$ dependence
compared to the earlier results, see Bombieri \cite{BOMBIERI}.

We will use our version of Siegel's lemma with Pad\'e-type approximations of the second kind to construct
simultaneous numerical linear forms over a ring $\mathbb{Z}_{\mathbb{I}}$.
At this point we apply the axiomatic method from  \cite{TM}.
Sequently, we achieve a fully Baker type lower bound, Theorem  \ref{BAKERTYPE},
for the exponential values  (\ref{1}) over  an arbitrary imaginary quadratic field, or the field of rational numbers, ${\mathbb{I}}$.

\section{Results and comparision}\label{Results}

\subsection{Results}

Let $\alpha_j =x_j/y_j \in \mathbb{I}$ be $m+1$ different numbers with 
$x_j \in \ZI$, $y_j \in \mathbb{Z}^+$ and $\gcd(x_j,y_j)=1$, when $j=0,1,\ldots,m$.
Denote
\begin{equation*}
\overline{\alpha}=(\alpha_0,\alpha_1,...,\alpha_m)^T,\quad
g_1(\overline{\alpha})=\lcm(y_0,\ldots,y_m),\quad 
g_2(\overline{\alpha})=\underset{0\le j\le m}\max\{|x_j|+|y_j|\},\quad
\end{equation*}
\begin{equation*}
g_3(\overline{\alpha})=\underset{0\le j\le m}\max|\alpha_j|,\quad 
g_4(\overline{\alpha})=\underset{1\le j\le m}\max\{1+\frac{|y_j|}{|x_j|}\}.
\end{equation*}
From the definitions we get
\begin{equation}\label{gggggg} 
2\le g_2,\quad \max\{g_4, 1+g_3\} \le g_2\le g_1(1+g_3)\le 2g_1\max\{1,g_3\},\quad g_1\le g_2^m,
\end{equation}
where $\overline{\alpha}=(0,\alpha_1,...,\alpha_m)^T\ne \overline{0}$ 
and $g_i=g_i(\overline{\alpha})$. Furthermore, write
\begin{equation*}
b_0=\sqrt{\log g_2}+\frac{\log g_4}{2\sqrt{\log g_2}},\quad 
e_0=3\sqrt{\log g_2}+\frac{\log{g_4}}{2\sqrt{\log g_2}},
\end{equation*}
and
\begin{alignat*}{1}
b_1&=\max\{0,\log g_1-\log g_2-\log g_4\};\\
e_1&=\max\{0,\log g_1+2\log(1+g_3)+2\log 2+1-\log g_2-\log g_4\}.
\end{alignat*}

Set now
\begin{alignat*}{3}
&A=b_0+e_0m,  \quad  &B=1+b_0+b_1+e_1m, \quad\quad  C=m,&\\
&D=b_0m+e_0m^2, \quad &E=(1+b_0+b_1)m+(2e_0+e_1)m^2. &
\end{alignat*}
Next we introduce the function 
$z:\mathbb{R}\to \mathbb{R}$, 
the inverse function of the function 
$y(z) = z \log z$, $z \geq e$, considered in detail in \cite{HANCLETAL}.
Using this function, we define
\begin{equation*}
\xi(z,H):=
A\left(2\frac{z(2 \log H)}{\log H}\right)^{1/2}+
B\frac{ z(2 \log H)}{\log H}+
C\frac{\log z(2 \log H)}{ \log H}
+D\frac{(\log z(2 \log H))^{1/2}}{\log H}.
\end{equation*}
Write then
\begin{equation}\label{valinta}
H_0=\max\{ e^{(\gamma\log\gamma)/2},2\log\frac{s_{\mathbb{I}}}{t_{\mathbb{I}}} \},
\quad \log\gamma=(3me_0)^2,
\end{equation}
where $s_{\mathbb{I}}$ and $t_{\mathbb{I}}$ are constants coming from Siegel's lemma, see formula (\ref{stsiegel}).

We are now ready to formulate the first theorem.
\begin{Theorem}\label{BAKERTYPE}
Let $m\ge 2$. With $\alpha_0=0$ and the notations above, we have
\begin{equation}\label{BAKERLOWER}
\left|\beta_0+\beta_1 e^{\alpha_1}+...+\beta_me^{\alpha_m}\right| > 
\frac{1}{2e^{E}H^{1+\epsilon(H)}}
\end{equation}
for all 
$$\overline{\beta}=(\beta_0,\beta_1,...,\beta_m)^T \in 
\mathbb{Z}_{\mathbb{I}}^{m+1}\setminus\{\overline{0}\}$$
and
\begin{equation}\label{HGREATERF}
H=\prod_{i=1}^{m}(2mH_i)\ge H_0,\quad  H_i\ge h_i=\max\{1,|\beta_i|\} 
\end{equation}
with an error term
\begin{equation}\label{AKTepsilon}
\epsilon(H)=\xi(z,H).
\end{equation}
\end{Theorem}

Consider the error term (\ref{AKTepsilon}) as a function of $H$. We can easily see that it is better than the corresponding term (\ref{Bepsilon}) given by Baker  \cite{BAKER}, 
and the corresponding terms in Mahler  \cite{MAHLER} and Sankilampi  \cite{SANKILAMPI}.
Namely, notice first that 
\begin{equation}\label{zfloghz2e}
z(2 \log H)< z_2(2 \log H):=\frac{2 \log H}{\log \frac{2 \log H}{\log(2 \log H)}}
\end{equation}
by (\ref{z1z3zz2z0}). 
Then, see  \cite{TM}, this may be further estimated by
\begin{equation}\label{EPSILON2xi2e}
z_2(2 \log H)\le 2\frac{\log(\gamma\log\gamma)}{\log(\gamma\log\gamma)-\log\log(\gamma\log\gamma)}
\left(1-\frac{\log 2}{\log (2\log H)}\right)\frac{\log H}{\log\log H}<
\end{equation}
\begin{equation}\label{WEAKERe}
2\rho\frac{\log H}{\log\log H},
\end{equation}
for all $H\ge H_0$ with $\rho=1.024$.

Hence, using the bound (\ref{WEAKERe}), we get the following corollary. It is of the same shape but weaker than (\ref{BAKERLOWER})
or the bounds obtained using (\ref{zfloghz2e}) or (\ref{EPSILON2xi2e}).

\begin{Corollary}\label{Corollary2.2}
If
$H\ge H_0,$
then
$$
\left|\beta_0+\beta_1 e^{\alpha_1}+...+\beta_me^{\alpha_m}\right| > 
$$
\begin{equation}\label{USUAL}
\frac{1}{2e^{E}(2\rho)^{C}}\left(\frac{\log\log H}{\log H}\right)^{C}
H^{-1-
\frac{2A\sqrt{\rho}}{\sqrt{\log\log H}}-
\frac{2B\rho}{\log  \log H}
-\frac{D}{\log H}\sqrt{\log\left( \frac{2\rho \log H}{\log\log H}\right)} }.
\end{equation}
\end{Corollary}

For later purposes we give the following result.

\begin{Corollary}\label{Corollary2.3}
With the notations above, we have
\begin{equation}\label{SENCE}
\left|\beta_0+\beta_1 e^{\alpha_1}+...+\beta_me^{\alpha_m}\right| > 
\frac{1}{\hat H^{1+\hat \epsilon(\hat H)}},\quad\textrm{with}\quad
\hat \epsilon(\hat H)=\hat A(\overline{\alpha})/\sqrt{\log\log \hat H}
\end{equation}
where
\begin{equation}\label{ahatA}
\hat A(\overline{\alpha})
\le 1+(3.036+7.084m)\sqrt{\log g_2}+0.633\sqrt{m}+0.580\sqrt{m}\sqrt{\log(1+g_3)},
\end{equation}
when $g_1(\overline{\alpha})\le g_2(\overline{\alpha})g_4(\overline{\alpha})$, and 
\[
\hat A(\overline{\alpha})
\le   1+(3.036+7.084m)\sqrt{\log g_2}+0.633\sqrt m +(0.290+0.410\sqrt{m})\sqrt{\log(g_1(1+g_3))}, \label{bhatA}
\]
when $g_1(\overline{\alpha})> g_2(\overline{\alpha})g_4(\overline{\alpha})$, 
for all
\begin{equation*}
\hat H:=\prod_{i=1}^{m}H_i\ge (2m)^{-m}H_0:=\hat H_0(\overline{\alpha}).
\end{equation*}
Further, $\log\hat H_0(\overline{\alpha})\le \frac{1}{2}(3me_0)^2e^{(3me_0)^2}$ if $2\log\left(2\log\frac{s_{\mathbb{I}}}{t_{\mathbb{I}}}\right)\le (3me_0)^2e^{(3me_0)^2}$.
\end{Corollary}

Corollary \ref{Corollary2.3} implies the next result where the lower bound is 
formally analogous to the bounds (\ref{Bakerresult}) and (\ref{Mahlerresult})
given by Baker and Mahler, respectively.

\begin{Corollary}\label{Corollary2.4}
Let $\gamma_0, \gamma_1,...,\gamma_m\in\mathbb{Q}$ be distinct rational numbers.
Then 
\begin{equation}\label{BMANALOG}
\left|\beta_0 e^{\gamma_0}+\beta_1 e^{\gamma_1}+...+\beta_m e^{\gamma_m}\right|
> \frac{M^{1-\hat \delta(M)}}{h_0h_1\cdots h_m},
\end{equation}
with $\hat \delta(M) \le \hat B(\overline{\gamma})/\sqrt{\log\log M}$ and
\begin{equation}\label{BMANALOGERROR}
\hat B(\overline{\gamma})
\le c_m m^2\sqrt{\log(g_1(\overline{\gamma})(1+g_3(\overline{\gamma})))},
\end{equation}
where $c_2=13$ and $c_m=12$, if $m\ge 3$, for all  $\overline{\beta}=(\beta_0,...,\beta_m)^T \in \mathbb{Z}^{m+1}\setminus\{\overline{0}\}$ and $M \ge M_{0,AKT}(\overline{\gamma})$ with
\begin{equation}\label{AKTHATM} 
\log M_{0,AKT}(\overline{\gamma})= 
96m^2(\log(g_1(\overline{\gamma})(1+g_3(\overline{\gamma})))) 
e^{192m^2\log(g_1(\overline{\gamma})(1+g_3(\overline{\gamma})))}.
\end{equation}
\end{Corollary}


\subsection{Comparisons to results of Mahler and Sankilampi}\label{Section2.2}

Here we will compare our(=AKT) results to the results given by Mahler(=MA) and Sankilampi(=SA).
Since the explicit results given by Mahler and Sankilampi are over $\mathbb{Q}$, in this chapter
we shall stay in the case $\mathbb{I}=\mathbb{Q}$. Write 
$g_1=g_1(\overline{\alpha})$ and $\tilde g_3=\tilde g_3(\overline{\alpha})=\max(1,g_3)$. 
Sankilampi \cite{SANKILAMPI} has a similar result as our Corollary \ref{Corollary2.3}
with a corresponding term
\begin{equation}\label{Sepsilon2}
\hat A_{SA}(\overline{\alpha})
= 16m^2+m\log(2m)+39m+12+\left(8+\frac{4}{m}+\frac{1}{3m^2}\right)
\log(2g_1\tilde g_3),
\end{equation}
(\cite{SANKILAMPI}, formula (71), page 32) valid for all $\hat H\ge\hat H_{0,SA}(\overline{\alpha})$, where
\begin{equation}\label{SHATH} 
\log\hat H_{0,SA}(\overline{\alpha})=
e^{(16m^2+36m+m\log(2m)+8\log(2g_1\tilde{g_3}))^2},
\end{equation}
(\cite{SANKILAMPI}, line 13, page 33).

\begin{Theorem}\label{vertailu} We have
\begin{equation}\label{AKTepsilon3}
\hat A_{AKT}(\overline{\alpha})\le \sqrt{m}+\left(4+\sqrt{m}+8m\right)\sqrt{\log(2g_1\tilde g_3)}
\end{equation}
for all $\hat H\ge\hat H_{0,AKT}(\overline{\alpha})$ with
\begin{equation}\label{AKTHATA} 
\log\hat H_{0,AKT}(\overline{\alpha})= 56m^2\log(2g_1\tilde g_3) e^{111m^2\log(2g_1\tilde g_3)}.
\end{equation}
Furthermore, we have $\hat A_{AKT}(\overline{\alpha})\leq \hat A_{SA}(\overline{\alpha})$.
\end{Theorem}

Following the notations in Corollary \ref{Corollary2.4} Mahler has a corresponding term
\begin{equation}\label{Mepsilon2}
\hat B_{MA}(\overline{\gamma})= 
12(m+1)^3\sqrt{\log(g_1(\overline{\gamma})(1+g_3(\overline{\gamma})))}
\end{equation}
valid for all $M\ge M_{0,MA}(\overline{\gamma})$ with
\begin{equation}\label{MHATM} 
\log M_{0,MA}(\overline{\gamma}) = 16(m+1)^4(\log(g_1(\overline{\gamma})+g_3(\overline{\gamma})))
 e^{16(m+1)^4\log (g_1(\overline{\gamma})(1+g_3(\overline{\gamma})))}.
\end{equation}
Our term 
\begin{equation*}
\hat B_{AKT}(\overline{\gamma})= 
c_m m^2\sqrt{\log(g_1(\overline{\gamma})(1+g_3(\overline{\gamma})))}
\end{equation*}
in Corollary \ref{Corollary2.4} is valid for all $M\ge M_{0,AKT}(\overline{\gamma})$ with
\begin{equation}\label{AKTHATM2} 
\log M_{0,AKT}(\overline{\gamma})= 
96m^2(\log(g_1(\overline{\gamma})(1+g_3(\overline{\gamma})))) 
e^{192m^2\log(g_1(\overline{\gamma})(1+g_3(\overline{\gamma})))}
\end{equation}
is clearly better in the dependence on $m$ than the term given by Mahler and 
in the lower bound the dependence on $m$ is improved from Mahler's quartic to our quadratic.  
 

\subsection{Examples}\label{Section2.3}

\begin{Example}\label{Example2.5} (The details of this example can be found in the Proofs section \ref{tokattodistukset}.) Choose $\alpha_j=j,\quad j=0,1,...,m.$ Then 
\begin{equation*}
\hat A_{AKT}(\overline{\alpha})\le 
1+0.670m+(2.252+6.072m)\sqrt{\log(m+1)}
\end{equation*}
and
\begin{equation*}
\log\hat H_{0,AKT}(\overline{\alpha})= 
(40.5m^2\log(m+1)+9.850)e^{81m^2\log(m+1)+19.699m^2}.
\end{equation*}
In particular, if $m=2$, then
\begin{equation*}
\hat A_{AKT}(\overline{\alpha})\le 18,\quad
\log\hat H_{0,AKT}(\overline{\alpha})=e^{441}
\end{equation*}
while the corresponding terms of Sankilampi are
\begin{equation*}
\hat A_{SA}(\overline{\alpha})\le 175,\quad
\log\hat H_{0,SA}(\overline{\alpha})=e^{23442}.
\end{equation*}
If we choose $\gamma_j=j,\quad j=0,1,...,m$, then our Baker-Mahler type terms are
\begin{equation*}
\hat B_{AKT}(\overline{\gamma})\le 
1+m+0.670m^2+(2.252m+6.072m^2)\sqrt{\log(m+1)}
\end{equation*}
and
\begin{equation*} 
\log M_{0,AKT}(\overline{\gamma})=(40.5m^2\log(m+1)+9.850)e^{81m^2\log(m+1)+19.699m^2}.
\end{equation*}
In particular, if $m=2$, then
\begin{equation*}
\hat B_{AKT}(\overline{\gamma}) \le 36,\quad
\log M_{0,AKT}(\overline{\gamma})= e^{441},
\end{equation*}
contra Mahler 
\begin{equation*}
\hat B_{MA}(\overline{\gamma}) \le 340,\quad
\log M_{0,MA}(\overline{\gamma}) = e^{1432}.
\end{equation*}
\end{Example}

\begin{Example}\label{Example2.6}
Choose
$\alpha_0=0,\quad \alpha_j=\frac{1}{j},\quad j=1,...,m.$ Using the bound $g_1=\lcm(1,2,...,m)\le e^{1.030883m}$ (see \cite{ROSSER}) it is straightforward to conclude that
\begin{equation*}
\hat A_{AKT}(\overline{\alpha})
\le 1+0.036m+(3.036+7.084m)\sqrt{\log(m+1)}
\end{equation*}
and
\begin{equation*}
\log\hat H_{0,AKT}(\overline{\alpha})
\le 55.125m^2(\log(m+1))e^{110.25m^2(\log(m+1))}. 
\end{equation*}
The corresponding terms of Sankilampi are
\[
\hat A_{SA}(\overline{\alpha})\le 16m^2+m\log m+47.941m+23.285,
\]
\[
\log\hat H_{0,SA}(\overline{\alpha})=e^{(16m^2+m\log m+44.998m+5.546)^2}.
\]
If we choose $\gamma_0=0,\quad \gamma_j=\frac{1}{j},\quad j=1,...,m$,
then our Baker-Mahler type terms are
\begin{equation*}
\hat B_{AKT}(\overline{\gamma}) \le 7.1m^2\sqrt{\log(m+1)}+...,\quad
\log M_{0,AKT}(\overline{\gamma})= 56m^2(\log(m+1))e^{111m^2\log(m+1)}
\end{equation*}
while Mahler has
\begin{equation*}
\hat B_{MA}(\overline{\gamma})=12.1m(m+1)^3,\quad
\log M_{0,MA}(\overline{\gamma}) = 16.1m(m+1)^4 e^{16.1m(m+1)^4}.
\end{equation*}
\end{Example}

In the last example we work on the field
$\mathbb{I}=\mathbb{Q}(\sqrt{-1})$ with the ring of integers
$\mathbb{Z}_{\mathbb{I}}=\mathbb{Z}[\sqrt{-1}]$, the ring of Gaussian integers.

\begin{Example}\label{Example2.7}
Let $r\ge \sqrt{2}$ be given and let
$\alpha_j \in \mathbb{Z}_{\mathbb{I}}$, $j=0,1,\ldots,m$, be $m+1$
integral points with absolute value at most $r$.
Then
\[
\hat A_{AKT}(\overline{\alpha}) \le 1+(1,596+4,294m)\sqrt{\log(m+1)}
\]
and
\begin{equation*}
\log\hat H_{0,AKT}(\overline{\alpha})= 
m^2(20,25\log(m+1)+9,604)e^{m^2(40.5\log(m+1)+19,207)}.
\end{equation*}
\end{Example}

\section{Proofs}

\subsection{Siegel's lemma}

Siegel's, or Thue-Siegel's, lemma is the following well-known result, used to bound the solutions of a Diophantine system of equations.

\begin{Lemma}[Siegel's lemma]\label{THUE}
Let $\mathbb{I}$ denote the field $\mathbb{Q}$ of rational numbers or
an imaginary quadratic field $\mathbb{Q(\sqrt{-D})}$, where $D\in \mathbb{Z}^+$  and 
$\ZI$ it's ring of intergers.\\
Let 
\begin{equation}\label{LINFORMS}
L_m(\overline{z})=\sum_{n=1}^{N} a_{mn}z_n,\quad m=1,...,M,
\end{equation}
be $M$ non-trivial linear forms with coefficients $a_{mn}\in \ZI$
in $N$ variables $z_n$.
Define $A_m:=\sum_{n=1}^{N}|a_{mn}|\in \mathbb{Z}^+$ for $m=1,...,M$. Suppose that $M<N$. Then there exists positive constants $s_{\mathbb{I}}, t_{\mathbb{I}}$ such that the system of equations  
\[
L_m(\overline{z})=0,\quad m=1,...,M,
\]
has a solution $\overline{z}=(z_1,...,z_N)^T\in \ZI^{N}\setminus\{\overline{0}\}$ with
\[
1\le \underset{1\le n\le N}\max|z_n|\le 
s_{\mathbb{I}} t_{\mathbb{I}}^{\frac{M}{N-M}}(A_1\cdots A_M)^{\frac{1}{N-M}},
\]
where $s_{\mathbb{Q}}=t_{\mathbb{Q}}=1$ (see e.g. \cite{MAHLER}) and $s_{\mathbb{I}}, t_{\mathbb{I}}$ suitable constants depending
on the field.
\end{Lemma}

We prove the following version of it:

\begin{Lemma}\label{Thuemax}
 There is a solution
\[
1\le \underset{1\le n\le N}\max|z_n|\le 
\max\left\{2c\sqrt{D},s_{\mathbb{I}} t_{\mathbb{I}}^{\frac{M}{N-M}}(A_1\cdots A_M)^{\frac{1}{N-M}}\right\},
\]
with 
\begin{equation}\label{stsiegel}
s_{\mathbb{Q(\sqrt{-D})}}=
\begin{cases}
&\frac{2\sqrt{2}D^{1/4}}{\sqrt{\pi}}; \\
&\frac{2}{\sqrt{\pi}}D^{1/4};\\
\end{cases}
\quad  t_{\mathbb{Q(\sqrt{-D})}}=
\begin{cases}
&\frac{5}{2\sqrt{2}};\\ 
&\frac{5}{2\sqrt{2}};\\ 
\end{cases}
\quad c=
\begin{cases}
&2\sqrt{2},\quad \text{if}\quad D\equiv 1,2\pmod{4};\\ 

&2,\quad \text{if}\quad D\equiv 3\pmod{4}.\\ 
\end{cases}
\end{equation}
\end{Lemma}

Notice that our constant $t_{\mathbb{I}}$ does not depend on $D$,
contrary to general results, see e.g. \cite{BOMBIERI}, where 
\[
t_{\mathbb{Q(\sqrt{-D})}}=2s_{\mathbb{Q(\sqrt{-D})}}=
\begin{cases}
2^8\sqrt{D},\quad D\equiv 1\ \text{or}\ 2\pmod{4};\\
2^7\sqrt{D},\quad D\equiv 3\pmod{4}.
\end{cases}
\]
On the other hand, we have to assume that the solution is always at least of the size $2c\sqrt{D}$. However, this assumption only affects small values of $A_i$. Asymptotically, this does not matter. Furthermore, one could choose the constant $c$ to be smaller than in this proof, but this would affect constants $t_{\mathbb{I}}$ and $s_{\mathbb{I}}$.

\begin{proof}

Assume $D\equiv 1,2\ (\bmod \ 4)$, and that $B\geq c\sqrt{D}$. In order to estimate the number of points in the ellipse $x^2+Dy^2\leq B^2$, allocate every integer point $(x_0,y_0)$ the square $\max(|x-x_0|,|y-y_0|)\leq \frac{1}{2}$. The number of points with integer coordinates in the ellipse $x^2+Dy^2\leq B^2$ is thus at least
\[
\pi\left(B-\frac{1}{\sqrt{2}}\right)\left(\frac{B}{\sqrt{D}}-\frac{1}{\sqrt{2}}\right)\geq  \pi B\left(\frac{B}{\sqrt{D}}-\sqrt{2}\right)\geq \pi B^2D^{-1/2}\left(1-\frac{\sqrt{2}}{c}\right).
\]
The number of values of the linear forms (\ref{LINFORMS}) with $|z_k|\le B$, $k=1,...,N$ is at most
\[
\prod_{i=1}^M\pi \left(A_iB+\frac{1}{\sqrt{2}}\right)\left(\frac{A_iB}{\sqrt{D}}+\frac{1}{\sqrt{2}}\right)
\leq B^{2M}D^{-M/2}\cdot \prod_{i=1}^M\left(\pi A_i^2\left(1+\frac{1}{c\sqrt{2}}\right)^2\right).
\]
Hence, the number of M-tuples with integer coordinates is  greater than the number of the values
of the linear forms  when
\[
\pi^N B^{2N}D^{-N/2}\left(1-\frac{\sqrt{2}}{c}\right)^N\geq \pi^M D^{-M/2}B^{2M}\cdot \left(1+\frac{1}{c\sqrt{2}}\right)^{2M}\prod_{i=1}^M A_i^2,
\]
i.e., when
\[
B\geq \frac{D^{1/4}}{\sqrt{\pi}}\cdot\left(1-\frac{\sqrt{2}}{c}\right)^{-N/(2N-2M)}\left(1+\frac{1}{c\sqrt{2}}\right)^{M/(N-M)} \left(\prod_{i=1}^M A_i\right)^{1/(N-M)}.
\]
Choose $c=2\sqrt{2}$. Now
\[
B\geq   \frac{\sqrt{2}D^{1/4}}{\sqrt{\pi}}\left(\frac{5}{2\sqrt{2}}\right)^{M/(N-M)} \left(\prod_{i=1}^M A_i\right)^{1/(N-M)}.
\]Let 
$$B'=\max\left(2\sqrt{2D},\frac{\sqrt{2}D^{1/4}}{\sqrt{\pi}}\left(\frac{5}{2\sqrt{2}}\right)^{M/(N-M)} \left(\prod_{i=1}^M A_i\right)^{1/(N-M)}\right).$$ 
By the pigeonhole principle, there is a solution when $|z_j|\leq 2B'$ for all $j$.

The proof is similar for $D\equiv 3 \ (\bmod \ 4)$, except that the integers in $\mathbb{Q}(\sqrt{D})$ are numbers of form $x+y\sqrt{D}$ where $x,y\in \mathbb{Z}$, when $D\equiv 1,2\ (\bmod \ 4)$ and $2x,2y\in \mathbb{Z}$, $2x\equiv 2y\ (\bmod \ 2)$ for $D\equiv 3\ (\bmod \ 4)$. Therefore, in order to count the number of points in an ellipse, we allocate every integer point $(x_0,y_0)$ the square consisting of numbers $(x,y)$ satisfying the condition $|x-x_0|+|y-y_0|\leq \frac{1}{2}$. The area of such a square is $\frac{1}{2}$. Furthermore, we choose $c=2$.\end{proof}

Taking into account that asymptotically the number of integers $(x,y)$ with $x^2+Dy^2\leq B^2$ is $\pi B^2D^{-1/2}$ when $D\equiv 1$ or $2\ (\bmod \ 4)$ and $2\pi B^2D^{-1/2}$, when $D\equiv 3\ (\bmod \ 4)$, we see that asymptotically, we have
\[
s_{\mathbb{Q(\sqrt{-D})}}=
\begin{cases}
&\frac{2}{\sqrt{\pi}}D^{1/4}; \\
&\frac{\sqrt{2}}{\sqrt{\pi}}D^{1/4};\\
\end{cases}
\quad  t_{\mathbb{Q(\sqrt{-D})}}=
\begin{cases}
& 1 \quad \text{if}\quad D\equiv 1,2\pmod{4}\\ 
& 1 \quad \text{if}\quad D\equiv 3\pmod{4};\\ 
\end{cases}
\]

\subsection{Construction of Hermite-Pad\'e type approximations}

The proof of the main Theorem \ref{BAKERTYPE} will be based on the 
construction the Hermite-Pad\'e type approximations 
(simultaneous functional approximations of type II)
for the vector  
$(1, e^{\alpha_1 t},...,e^{\alpha_m t})$.

Let $\nu_1,\ldots,\nu_m,l_1,\ldots,l_m \in \mathbb{Z}^+$, 
$\nu=\max\{\nu_1,...,\nu_m\}$ satisfy
\[
1 \le \nu_j\le l_j,\qquad M=\nu_1+\cdots+\nu_m\le L=l_1+\cdots+l_m.
\]
Write
\[
\overline{\alpha}=(\alpha_{1},...,\alpha_{m})^T,\quad 
\overline{\nu}(t)=(\nu_{1},...,\nu_{m})^T.
\]

\begin{Lemma}\label{existence}
There exists a non-zero polynomial 
\[
A_{0,0}(t) = \sum_{h=0}^L c_h \frac{L!}{h!}t^h \in \mathbb{Z}_{\mathbb{I}}[t]
\]
depending on $L$. Furthermore, there exist non-zero polynomials 
$A_{0,j}(t) \in \mathbb{Z}_{\mathbb{I}}[t,\overline{\alpha}]$, $j=1,\ldots,m$,
depending on $L$ and $\overline{\nu}$ ,
such that
\begin{equation}\label{2pade}
A_{0,0}(t)e^{\alpha_j t}+A_{0,j}(t)=
R_{0,j}(t), \qquad j=1,\ldots,m,
\end{equation}
where
\begin{equation}\label{asteorder}
\begin{cases}
\deg_t A_{0,j}(t) \le L, \qquad j=0,\ldots,m, \\
L+\nu_j+1\le\underset{t=0}\ord R_{0,j}(t) <\infty, \qquad j=1,\ldots,m.
\end{cases}
\end{equation}
Moreover, we have
\begin{equation}\label{chSiegel}
|c_h|\le \max\left\{2c\sqrt{D},
s_{\mathbb{I}} t_{\mathbb{I}}^{\frac{M}{L+1-M}} \left(g_2^{ML}g_4^{M^2/2}\right)^{\frac{1}{L+1-M}}\right\},
\quad h=0,1,...,L.
\end{equation}
\end{Lemma}

\begin{proof} Let
\[
A_{0,0}(t)e^{\alpha_j t}=\sum_{N=0}^{\infty}r_{N,j}t^N, \qquad j=1,...,m,
\]
where $r_{N,j}=\sum_{N=h+n}c_h\frac{L!}{h!} \frac{\alpha_j^n}{n!}$. Write now
\[
A_{0,j}(t)= \sum_{N=0}^Lr_{N,j}t^N, \qquad j=1,...,m.
\]
Notice that if $c_h \in \mathbb{Z}_{\mathbb{I}}$ for $h=0, \ldots, L$, then $g_1^LA_{0,j}(t)\in \mathbb{Z}_{\mathbb{I}}[t]$ for $j=1,\ldots,m$. Set
\begin{equation}\label{rLkiszero}
r_{L+i_j,j}=0
\end{equation}
for $i_j=1,...,\nu_j$, $j=1,...,m$.
Then we multiply equation \eqref{rLkiszero} by $\frac{(L+i_j)!}{L!} y_j^{L+i_j}x_j^{-i_j}$, which gives $M$ equations
\begin{equation}\label{EQINT}
\sum_{\underset{0\le h\le L}{L+i_j=h+n}} \binom{L+i_j}{h}x_j^{L-h}y_j^h c_h=0,\quad  i_j=1,...,\nu_j, j=1,...,m,
\end{equation}
with integer coefficients in $L+1$ unknowns $c_h$, $0\le h\le L$. Further, the coefficients in (\ref{EQINT}) satisfy
\[
|x_j|^{-i_j}\sum_{\underset{0\le h\le L}{L+i_j=h+n}} \left|\binom{L+i_j}{h}x_j^{L+i_j-h}y_j^h\right|<
\left(|x_j|+|y_j|\right)^{L}\left(1+\frac{|y_j|}{|x_j|}\right)^{i_j},\quad  i_j=1,...,\nu_j, j=1,...,m.
\]
Thus, by Siegel's lemma, Lemma \ref{Thuemax}, there exists a solution $(c_0,\ldots,c_L)^T\in\mathbb{Z}_{\mathbb{I}}^{L+1}\setminus\{ \overline{0}\}$ with 
\[
|c_h|\le \max\left\{2c\sqrt{D}, 
s_{\mathbb{I}} t_{\mathbb{I}}^{\frac{M}{L+1-M}}
\left(\Pi_1\Pi_2\right)^{\frac{1}{L+1-M}}\right\},
\]
where
\begin{alignat*}{1}
\Pi_1&=\prod_{i_j,j} \left(|x_j|+|y_j|\right)^{L}=\prod_{j} \left(|x_j|+|y_j|\right)^{\nu_jL}\le g_2^{ML}\\
\Pi_2&=\prod_{i_j,j}\left(1+\frac{|y_j|}{|x_j|}\right)^{i_j}\le 
\prod_{j}\left(1+\frac{|y_j|}{|x_j|}\right)^{ \frac{\nu_j(\nu_j+1)}{2}} \le g_4^{M^2/2}.
\end{alignat*}
Writing
\[
R_{0,j}(t) = \sum_{N=L+\nu_j+1}^{\infty} r_{N,j} t^{N},\qquad j=1,\ldots,m,
\]
we get
\begin{equation}\label{pade20}
A_{0,0}(t)e^{\alpha_j t}+A_{0,j}(t)=R_{0,j}(t),\qquad j=1,\ldots,m,
\end{equation}
where $A_{0,j}(t)$ are non-zero polynomials for all $j=0,1,...,m$ and
\[
\begin{cases}
\deg_tA_{0,j}(t) \le L, \qquad  j=0,\ldots,m, \\
L+\nu_j+1\le \underset{t=0}\ord R_{0,j}(t) <\infty, \qquad j=1,\ldots,m,
\end{cases}
\]
which means that we have Hermite-Pad\'e type approximations for the vector  
$(1, e^{\alpha_1 t},...,e^{\alpha_m t})$.\end{proof}

Now we need to construct more linear forms. For this purpose, we define recursively
\[
\begin{cases}
A_{k+1,j}(t)=A'_{k,j}(t)-\alpha_jA_{k,j}(t), \qquad j = 0,\ldots,m, \\
R_{k+1,j}(t)=R'_{k,j}(t)-\alpha_jR_{k,j}(t), \qquad j = 1,\ldots,m.
\end{cases}
\]
By \eqref{pade20} and repeated application of the differential operator $D=\tfrac{d}{dt}$ we get 
\begin{equation}\label{pade2k}
A_{k,0}(t)e^{\alpha_jt}+A_{k,j}(t)=R_{k,j}(t), \qquad j=1,\ldots,m,\quad k=0,1,... .
\end{equation}
Further, we have
\begin{equation}\label{INEQ}
\begin{cases}
\deg_tA_{k,0}(t) = \deg_tA_{0,0}(t)-k, \\
\deg_tA_{k,j}(t) = \deg_tA_{0,j}(t), \qquad j=1,\ldots,m, \\
\underset{t=0}\ord R_{k,j}(t) = \underset{t=0}\ord R_{0,j}(t)-k, \qquad j=1,2,\ldots,m. 
\end{cases}
\end{equation}

These linear forms need to be linearly independent, and therefore, we now consider the determinant of a matrix defined by
\begin{equation}\label{MATRIXA}
\mathcal{A}(t) =
\left(\begin{matrix}
 A_{0,0}(t) & A_{0,1}(t) & \cdots & A_{0,m}(t) \\
 A_{1,0}(t) & A_{1,1}(t) & \cdots & A_{1,m}(t) \\
\vdots & \vdots & & \vdots \\
 A_{m,0}(t) & A_{m,1}(t) & \cdots & A_{m,m}(t)
\end{matrix}\right)
\end{equation}
and denote its  determinant by
$\Delta(t) =\det\mathcal{A}(t) $.
Note that
\[
\Delta(t) =\left|\begin{matrix}
A_{0,0}(t) & R_{0,1}(t) & \cdots & R_{0,m}(t) \\
A_{1,0}(t) & R_{1,1}(t) & \cdots & R_{1,m}(t) \\
\vdots & \vdots & & \vdots \\
A_{m,0}(t) & R_{m,1}(t) & \cdots & R_{m,m}(t)
\end{matrix}\right|.
\]

\begin{Lemma}\label{DETERMINANTTI}
There exist $m+1$ different
indices $s(0), s(1), \ldots,s(m)$ in the set $0,1,\ldots,L-M+m(m+1)/2$ such that the determinant
\[
\overline{\Delta}(1) = \left|\begin{matrix}
A_{s(0),0}(1) & A_{s(0),1}(1) & \cdots & A_{s(0),m}(1) \\
A_{s(1),0}(1) & A_{s(1),1}(1) & \cdots & A_{s(1),m}(1) \\
\vdots & \vdots &  & \vdots \\
A_{s(m),0}(1) & A_{s(m),1}(1) & \cdots & A_{s(m),m}(1)
\end{matrix} \right|
\]
is nonzero.
\end{Lemma}

\begin{proof}
By \eqref{asteorder}  and \eqref{INEQ}, we have $\deg_t \Delta(t) \le (m+1)L$ and  $\underset{t=0}\ord \Delta(t) \ge mL+M-m(m-1)/2$. Hence
\begin{equation}\label{DETPOL}
\Delta(t) = t^{mL+M-m(m-1)/2}h(t),\quad 
\deg(h(t)) \le a:=L-M+m(m-1)/2.
\end{equation}
If we now denote the highest degree coefficients of the polynomials 
$A_{0,0}(t)$ and $A_{0,j}(t)$ by $a_{0,0}$ and $a_{0,j}$, respectively, then 
the highest degree coefficients of the polynomials $A_{k,0}(t)$ and 
$A_{k,j}(t)$ are 
\[
a_{0,0}\prod_{i=0}^{k-1}(\deg(A_{0,0}(t))-k+1+i)\quad\text{and}\quad a_{0,j}{(-\alpha_j)}^k,
\]
respectively. The highest degree coefficient of the polynomial $\Delta(t)$ is therefore
\[
(-1)^{\lfloor m/2 \rfloor}a_{0,0}a_{0,1}\cdots a_{0,m}
\left|
\begin{matrix}
\alpha_1 & \alpha_2 & \cdots & \alpha_m \\
\alpha_1^2 & \alpha_2^2 & \cdots &\alpha_m^2 \\
\vdots & \vdots & & \vdots \\
\alpha_1^m & \alpha_2^m & \cdots & \alpha_m^m
\end{matrix}
\right| \neq 0,
\]
for the numbers $\alpha_j$ are different. Thus, $\Delta(t)$ is a nonzero polynomial. By (\ref{DETPOL}) there exists an integer $0\le b\le  a$ such that $\Delta^{(b)}(1) \ne 0,\quad \Delta^{(c)}(1) =0$ for all $c<b$.

Define the linear combinations
\begin{equation}\label{WEXP}
w_i(t) = A_{i,0}(t)e^{-\alpha_0 t}+A_{i,1}(t)e^{-\alpha_1 t}+\cdots+A_{i,m}(t)e^{-\alpha_m t}
\end{equation}
for $i = 0,1,\ldots,m$.
Derivation with respect to $t$ gives
\begin{equation}\label{Wderivaatta}
w_i'(t)  = A_{i+1,0}(t)e^{-\alpha_0 t}+A_{i+1,1}(t)e^{-\alpha_1 t}+\cdots+A_{i+1,m}(t)e^{-\alpha_m t} =
 w_{i+1}(t)
\end{equation}
by the definition of the polynomials $A_{i,j}$.

Next we define a matrix 
\begin{equation}\label{MATRIXC}
\mathcal{C}(t) =\left(\Delta_{i,j}(t)\right)_{i,j=0,1,...,m}
\end{equation}
by the cofactors $\Delta_{i,j}(t)$ of $\mathcal{A}(t)$ satisfying
\begin{equation}\label{MATRIXAC}
\mathcal{A}(t)\mathcal{C}(t) =\mathcal{C}(t)\mathcal{A}(t)=\Delta(t).
\end{equation}
Write
\[
\overline{w}_k(t)=(w_{k}(t),w_{k+1}(t),...,w_{k+m}(t))^T,\quad \overline{e}(t)=
(e^{-\alpha_0 t},e^{-\alpha_1 t},...,e^{-\alpha_m t})^T.
\]
By \eqref{Wderivaatta} we see that $\overline{w}'_{k}(t)=\overline{w}_{k+1}(t)$ and by \eqref{WEXP} we may write $\overline{w}_{0}(t)=\mathcal{A}(t)\overline{e}(t)$ yielding
\[
\Delta(t)\overline{e}(t)= \mathcal{C}(t)\overline{w}_{0}(t).
\]
Differentiating $b$ times gives
\[
\sum_{k=0}^{b}{b\choose k}\Delta^{(k)}(t)\overline{e}^{(b-k)}(t)=
 \sum_{k=0}^{b}{b\choose k}\mathcal{C}^{(k)}(t)\overline{w}^{(b-k)}_{0}(t),
\]
and we obtain
\[
\Delta^{(b)}(1)\overline{e}(1)=
 \sum_{k=0}^{b}{b\choose k}\mathcal{C}^{(k)}(1)\overline{w}_{b-k}(1).
\]

Thus, the $m+1$ linearly independent numbers $1=e^{-\alpha_0}, e^{-\alpha_1},\ldots,e^{-\alpha_m}$ 
can be represented by linear combinations of the $m+b+1$ numbers 
$w_0(1), w_1(1), \ldots, w_{m+b}(1)$. Therefore, there exist $m+1$ different
indices $s(0), s(1), \ldots,s(m)$ in the set $0,1,\ldots,m+b$ such that the determinant
\[
\overline{\Delta}(1) = \left|\begin{matrix}
A_{s(0),0}(1) & A_{s(0),1}(1) & \cdots & A_{s(0),m}(1) \\
A_{s(1),0}(1) & A_{s(1),1}(1) & \cdots & A_{s(1),m}(1) \\
\vdots & \vdots &  & \vdots \\
A_{s(m),0}(1) & A_{s(m),1}(1) & \cdots & A_{s(m),m}(1)
\end{matrix} \right|
\]
is nonzero. This completes the proof of the lemma.

\end{proof}

\subsection{Numerical linear forms}

In order to obtain good approximations for linear forms, we need good upper bounds for the coefficients and remainders that are used. Therefore, we consider the numerical linear forms $B_{k,0} e^{\alpha_j}+ B_{k,j}  = L_{k,j}$ with $j=1,...,m$, where
\[
B_{k,j}(t) := g_1^L A_{s(k),j}(t),\quad B_{k,j}=B_{k,j}(1),\quad j=0,1,...,m
\]
and
\[
L_{k,j}(t)   := g_1^L R_{s(k),j}(t),\quad L_{k,j}=L_{k,j}(1),\quad j=1,...,m.
\]

The following lemma tells that the linear forms $L_{k,j}=B_{k,0} e^{\alpha_j}+ B_{k,j}$
have coefficients $B_{k,j}\in \mathbb{Z}_{\mathbb{I}}$. Furthermore,  it gives necessary estimates for the 
coefficients   $B_{k,0}$ and the remainders $L_{k,j}$.

\begin{Lemma}\label{Estimates}
For all $k\in\mathbb{N}$ we have $B_{k,j}  \in\mathbb{Z}_{\mathbb{I}}$ when $j=0,1,...,m$,
\begin{equation}\label{Upperboundden}
|B_{k,0}| \le eg_1^L L! \max\{|c_{h}|\};
\end{equation}
\begin{equation}\label{Upperboundrem}
|L_{k,j}|\le g_1^L(1+|\alpha_j|)^{L+\nu_j+1-s(k)}e^{1+|\alpha_j|}
\frac{L!}{(L+\nu_j+1-s(k))!}
\max\{|c_{h}|\},\quad j=1,...,m,
\end{equation}
and
\begin{equation}\label{skleLM}
s(k)\le L-M+m(m+1)/2.
\end{equation}
\end{Lemma}

\begin{proof} First we notice that
\[
B_{k,0}(t)=g_1^LD^kA_{0,0}(t)=
g_1^L\sum_{h=0}^{L-k} c_{h+k} \frac{L!}{h!} t^h\in\mathbb{Z}_{\mathbb{I}}[t]
\]
and
\[
B_{k,j}(t)=-[B_{k,0}(t)e^{\alpha_j t}]_L=\sum_{N=0}^{L} b_{N,j}t^N, \qquad j=1,...,m,
\]
where
\[
b_{N,j}=-g_1^L\sum_{N=h+n}c_{h+k}\frac{L!}{h!} \frac{\alpha_j^n}{n!}\in\mathbb{Z}_{\mathbb{I}}.
\]
Thus the $B_{k,j}\in \mathbb{Z}_{\mathbb{I}}$. We may now turn to the next claim. The simple estimate
\[
|B_{k,0}(t)|\le g_1^L L!\underset{0\le h\le L}\max\{|c_{h}|\} 
\sum_{h=0}^{L-k}  \frac{|t|^h}{h!}
\]
proves (\ref{Upperboundden}). 
For the remainders, we have $L_{k,j}(t)=\sum_{N=L+\nu_j+1-s(k)}^{\infty} b_{N,j}t^N$, where
\[
|b_{N,j}|=g_1^L\frac{L!}{N!}\left|\sum_{N=h+n}c_{h+k}{N\choose h} \alpha_j^n\right|\le
g_1^L\frac{L!}{N!}
\underset{0\le h\le L}\max\{|c_{h}|\}
(1+|\alpha_j|)^{N}.
\]
Set $V=L+\nu_j+1-s(k)$, then 
\[
|L_{k,j}(t)|\le g_1^L \frac{L!}{V!}
\underset{0\le h\le L}\max\{|c_{h}|\}
((1+|\alpha_j|)|t|)^{V}
\sum_{i=0}^{\infty}\frac{V!}{(V+i)!}(1+|\alpha_j|)^{i}|t|^i.\]
This completes the proof of the lemma. \end{proof}

Now we choose the parameters $\nu_j$ in an appropriate way. Let
\begin{equation}\label{CHOOSEPARAMETERS}
\nu_j =\left\lfloor l_j\left(1-\sqrt\frac{\log g_2}{\log L}\right)\right\rfloor,\quad g_2\le L.
\end{equation}
Then we have
\begin{equation}\label{LMmL}
L\left(1-\sqrt\frac{\log g_2}{\log L}\right)-m<
M \le L\left(1-\sqrt\frac{\log g_2}{\log L}\right).
\end{equation}
Furthermore, notice that from $M\leq L\left(1-\sqrt{\frac{\log g_2}{\log L}}\right)$, we get
\begin{equation}\label{fracMLM}
\frac{M}{L+1-M}<\frac{M}{L-M}\leq  \sqrt{\frac{\log L}{\log g_2}} -1.
\end{equation}
Using the inequality $M\leq L\left(1-\sqrt{\frac{\log g_2}{\log L}}\right)$, we obtain
\begin{equation}\label{fracM2LM}
\frac{M^2/2}{L+1-M}<\frac{M^2/2}{L-M}\leq \frac{\left(L\left(1-\sqrt{\frac{\log g_2}{\log L}}\right)\right)^2}{L-L\left(1-\sqrt{\frac{\log g_2}{\log L}}\right)}=\frac{1}{2}L\sqrt{\frac{\log L}{\log g_2}}-L+\frac{L}{2}\sqrt{\frac{\log g_2}{\log L}}.
\end{equation}

\begin{Lemma}\label{Estimates2}
Suppose
\begin{equation}\label{ASS860}
L\ge\max\{ e^{(\gamma\log\gamma)/2},2\log\frac{s_{\mathbb{I}}}{t_{\mathbb{I}}} \},
\quad \log\gamma=(3me_0)^2.
\end{equation}
Then we have
\begin{equation}\label{AeqN}
|B_{k,0}|\le e^{q(L)},
\end{equation}
\begin{equation}\label{Lerin}
|L_{k,j}|\le  e^{-r_j(\overline{l})},
\end{equation}
where
\begin{equation}\label{qN2}
q(L)=aL\log L+b_0L(\log L)^{1/2}+b_1L,
\end{equation}
\begin{equation}\label{rin2}
-r_j(\overline{l})=(dL-cl_j)\log L+e_0L(\log L)^{1/2}+e_1L,
\end{equation}
for all $j=1,...,m$ and $k=0,1,...,m$, and with
\[\begin{cases}
&b_0=\sqrt{\log g_2}+\frac{\log g_4}{2\sqrt{\log g_2}}\\
& e_0=3\sqrt{\log g_2}+\frac{\log{g_4}}{2\sqrt{\log g_2}}\\
&b_1=\max\{0,\log g_1-\log g_2-\log g_4\}\\
&e_1=\max\{0,\log g_1+2\log(1+g_3)+2\log 2-\log g_2-\log g_4+1\},
\end{cases}
\]
\end{Lemma}

\begin{proof} 
By (\ref{gggggg}) and the definition of $e_0$ we have $e_0^2 \ge 9 \log 2$.
From the assumption (\ref{ASS860}) and the fact that $m\ge 2$ we now get
\begin{equation}\label{lallallaa}
\log L \ge \frac{1}{2} \gamma \log \gamma
= \frac{1}{2} e^{9m^2e_0^2}9m^2e_0^2>e^{229}.
\end{equation}
To simplify the notation, define
\begin{equation*}
p=\frac{m(m+3)}{2},
\end{equation*}
and notice that we have
\begin{equation}\label{coming6}
p \le \frac{5}{4}m^2 \le \log \log L.
\end{equation}

Recall Stirling's formula (see e.g. \cite{stirling})
\[
n! = \sqrt{2\pi}n^{n+\frac{1}{2}}e^{-n+\frac{\theta(n)}{12n}}, \quad  0<\theta(n)<1,
\]
which we use in the form
\[
\log n!=\log \sqrt{2\pi}+z(n)+\frac{\theta(n)}{12n},
\]
where we denoted $z(n)= (n+1/2)\log n-n$.
We are now ready to estimate $B_{k,0}$.
From (\ref{Upperboundden})  and  (\ref{chSiegel}), we get
\begin{equation}\label{matters}
|B_{k,0}| \le eg_1^L L! 
\max\left\{2c\sqrt{D},
s_{\mathbb{I}} t_{\mathbb{I}}^{\frac{M}{L+1-M}} \left(g_2^{ML}g_4^{M^2/2}\right)^{\frac{1}{L+1-M}}\right\}.
\end{equation}
Use now the estimates (\ref{fracMLM})-(\ref{fracM2LM}) to obtain
\[
|B_{k,0}| \le eg_1^L L! 
\max\left\{2c\sqrt{D},
s_{\mathbb{I}} t_{\mathbb{I}}^{\sqrt{\frac{\log L}{\log g_2}}-1} g_2^{L\sqrt{\frac{\log L}{\log g_2}}-L}g_4^{\frac{L}{2}\sqrt{\frac{\log L}{\log g_2}}-L+\frac{L}{2}\sqrt{\frac{\log g_2}{\log L}}}\right\}\]
Next we note that by (\ref{ASS860}), we have $\sqrt{D}\le \frac{\pi\cdot 25^2}{2\cdot 4^2}e^L\leq 2.5e^L$. Since $L\geq e^{\gamma\log \gamma}$ and $\log \gamma=(3me_0)^2\geq 81m^2\log g_2$, it is clear that the second term dominates, and we can conclude that
\begin{multline*}
\log|B_{k,0}| \le 1+L\log g_1+\log(\sqrt{2\pi})+ \left(L+\frac{1}{2}\right)\log L-L+\frac{1}{12L}
+\log s_{\mathbb{I}}\\+\left(\sqrt{\frac{\log L}{\log g_2}}-1\right)\log t_{\mathbb{I}}
+\left(L\sqrt{\frac{\log L}{\log g_2}}-L\right)\log g_2 +\left(\frac{1}{2}L\sqrt{\frac{\log L}{\log g_2}}-L+\frac{1}{2}L\sqrt{\frac{\log g_2}{\log L}}\right)\log g_4
\end{multline*}
Recalling that $\log t_{\mathbb{I}}\leq 2$, $g_2\geq 2$, and noticing that (\ref{lallallaa}) implies that $\frac{\log g_4}{2}\sqrt{\frac{\log g_2}{\log L}} \le \frac{1}{24}$, we may further estimate this expression by
\begin{multline*}
\le L\log L+\left(\sqrt{\log g_2}+\frac{\log g_4}{2\sqrt{\log g_2}}\right)L\sqrt{\log L} +\left(\log g_1-\log g_2-\log g_4\right)L\\+
\left(-1+\frac{1}{24}\right)L+\frac{1}{2}\log L+\frac{2}{ \sqrt{\log 2} }\sqrt{\log L}+
1+ \log \sqrt{2 \pi}+\frac{1}{12L}+\log s_{\mathbb{I}}-\log t_{\mathbb{I}},
\end{multline*}
We will now show that the expression on the second line $\leq 0$. First, by (\ref{ASS860}), we have,
\[
-\frac{1}{2}L + \log s_{\mathbb{I}}-\log t_{\mathbb{I}} \le 0.
\]
Secondly, for large enough $L$ (recall that we have $L\geq \exp(e^{229})$, which is large enough), we have
\[
-\frac{11}{24}L+\frac{1}{2}\log L + \frac{2}{\sqrt{\log 2}}\sqrt{\log L}+1+\log{\sqrt{2\pi}}+\frac{1}{12L} \le 0,
\]
and hence, we may conclude
\[
\log|B_{k,0}|\le L\log L+\left(\sqrt{\log g_2}+\frac{\log g_4}{2\sqrt{\log g_2}}\right)L\sqrt{\log L} +\left(\log g_1-\log g_2-\log g_4\right)L.
\]
This proves (\ref{qN2}), i. e., we have now derived that we can choose
\[
b_0=\sqrt{\log g_2}+\frac{\log g_4}{2\sqrt{\log g_2}}\quad \textrm{and}\quad b_1=\max\{0,\log g_1-\log g_2-\log g_4\}.
\]
The reason why we set $b_1$ to be the maximum of $0$ and $\log g_1-\log g_2-\log g_4$ will be later motivated.

We may now turn to bounding $L_{k,j}$. Using the bound (\ref{Upperboundrem}) for $L_{k,j}$ and bound (\ref{chSiegel}) for $c_h$, together with the estimate
\[
\frac{L!}{(L+\nu_j+1-s(k))!} \le \frac{s(k)!L!}{(L+\nu_j)!}\binom{L+\nu_j}{s(k)}
\]
and Stirling's formula, we get
\begin{multline*}
\log |L_{k,j}| \le L\log g_1+(L+\nu_j+1-s(k))\log(1+|\alpha_j|)+1+|\alpha_j|
+ \log(\sqrt{2\pi})+z(s(k))+\frac{1}{12s(k)}\\+z(L)+\frac{1}{12L}-z(L+\nu_j)+(L+\nu_j)\log 2+\log s_{\mathbb{I}}+\left(\sqrt{\frac{\log L}{\log g_2}}-1\right)\log t_{\mathbb{I}}
\\+\left(L\sqrt{\frac{\log L}{\log g_2}}-L\right)\log g_2 +\left(\frac{1}{2}L\sqrt{\frac{\log L}{\log g_2}}-L+\frac{1}{2}L\sqrt{\frac{\log g_2}{\log L}}\right)\log g_4,
\end{multline*}
where $z(t)=t\log t-t+\frac{1}{2}\log t$. We wish to estimate this expression for $\log |L_{k,j}|$. Start with the function $z(t)$. For $t>1/2$ the function $z(t)$ has a positive increasing derivative 
$z'(t)=\log t+1/(2t)$. Thus,
\begin{multline*}
z(L)-z(L+\nu_j)\le -\nu_j\left(\log L+\frac{1}{2L}\right)
\le -\left(l_j\left(1-{\sqrt{\frac{\log g_2}{\log L}}}\right)-1\right)\log L\\\le -l_j\log L+
\sqrt{\log g_2}L\sqrt{\log L}+\log L.
\end{multline*}
By (\ref{skleLM}), we have
\[
s(k)\le L-M+\frac{m(m+1)}{2} \le \sqrt{\log g_2}\frac{L}{\sqrt{\log L}}+\frac{m(m+3)}{2}:=K+p.
\]
Since the function $\log t$ has a positive decreasing derivative $1/t$ when $t>0$, we get
\begin{multline*}
z(s(k))\le z(K+p)\le z(K)+p\left(\log(K+p)+\frac{1}{2(K+p)}\right)\\\le z(K)+p \log K+\frac{p^2}{K}+\frac{p}{2(K+p)}
\le K\log K -K+\frac{1}{2}\log K+ (\log K)^2+\frac{(\log K)^2}{K}+\frac{\log K}{2K+10},
\end{multline*}
where we used the estimates $5\le p \le \log K$, implied by (\ref{coming6}).
Since $L \ge g_2$ we have $K\log K \le \sqrt{\log g_2}L\sqrt{\log L}$, and 
\[
-K + \frac{1}{2}\log K + (\log K)^2+ \frac{(\log K)^2}{K}+\frac{\log K}{2K+10} \le 0
\]
whenever $K \ge 1$, which is true whenever $L \ge 2$ because $K \ge \sqrt{\log 2/\log L}L$. Thus we have
\[
z(s(k)) \le \sqrt{\log g_2} L \sqrt{\log L}.
\]

By combining the above estimates while bounding $\nu_j \le L$, $s(k) \ge 1$ and  $|\alpha_j| \le g_3$, we get
\begin{multline*}
\log |L_{k,j}|  \le -l_j\log L+\left(3\sqrt{\log g_2}+\frac{\log{g_4}}{2\sqrt{\log g_2}}\right)L\sqrt{\log L}\\
+\left(\log g_1+2\log(1+g_3)+2\log 2-\log g_2-\log g_4+\frac{\log g_4}{2}\sqrt{\frac{\log g_2}{\log L}}\right)L\\
+\frac{13}{12}+g_3+\log(\sqrt{2\pi})+\frac{1}{12L}+\log t_{\mathbb{I}}\sqrt{\frac{\log L}{\log g_2}}+\log L
+\log s_{\mathbb{I}}-\log t_{\mathbb{I}}.
\end{multline*}
As earlier, we deduce
\begin{multline*}
\log |L_{k,j}|  \le -l_j\log L+\left(3\sqrt{\log g_2}+\frac{\log{g_4}}{2\sqrt{\log g_2}}\right)L\sqrt{\log L}\\+\left(\log g_1+2\log(1+g_3)+2\log 2-\log g_2-\log g_4
+1\right)L\\\left(-1+\frac{1}{24}\right)L+\frac{1}{12}+2\log L+\log(\sqrt{2\pi})+\frac{1}{12L}+
\frac{2}{\sqrt{\log 3}}\sqrt{\log L}+\log s_{\mathbb{I}}-\log t_{\mathbb{I}}.
\end{multline*}
We will now prove that the last line is at most zero, by noting that
\[
-\frac{11}{24}L+\frac{1}{12}+2\log L+\log(\sqrt{2\pi})+\frac{1}{12L}+
\frac{2}{\sqrt{\log 2}}\sqrt{\log L} \le 0
\]
which holds when $L$ is large enough, and again we have 
$-\frac{1}{2}L + \log s_{\mathbb{I}}-\log t_{\mathbb{I}} \le 0$, too.
Hence
\begin{multline*}
\log |L_{k,j}|  \le -l_j\log L+\left(3\sqrt{\log g_2}+\frac{\log{g_4}}{2\sqrt{\log g_2}}\right)L\sqrt{\log L}\\
+\left(\log g_1+2\log(1+g_3)+2\log 2-\log g_2-\log g_4+1\right)L
\end{multline*}
and we may choose
\begin{alignat*}{1}
e_0&=3\sqrt{\log g_2}+\frac{\log{g_4}}{2\sqrt{\log g_2}}\\
e_1&=\max\{0,\log g_1+2\log(1+g_3)+2\log 2-\log g_2-\log g_4+1\},
\end{alignat*}
where the choice of $e_1$ will be later motivated. This completes the proof of Lemma \ref{Estimates2}.\end{proof}

\begin{Lemma}\label{Estimates3}
We have
\[
\log\gamma=(3me_0)^2\ge \log S_2
\]

where $S_2$ denotes the largest solution of the equation
\begin{equation}\label{gammaehto}
S\log S= 2(e_0mS(\log S)^{1/2}+ e_1mS+e_0 m^2(\log S)^{1/2}+2e_0m^2+e_1m^2).
\end{equation}
\end{Lemma}
\begin{proof}
First we shall prove
\begin{equation}\label{12me1}
25me_1\le (3me_0)^2.
\end{equation}
Starting from the definition and using (\ref{gggggg}) we get
\begin{multline*}
e_1=\max\{0,\log g_1+2\log(1+g_3)+2\log 2+1-\log g_2-\log g_4\}\\ 
\le \log g_1+\log(1+g_3)+2\log 2+1 \le (m+3+\frac{1}{\log 2})\log g_2 \le 3.222m \log g_2,
\end{multline*}
which is a very rough estimate but sufficient for our purposes here. Now, to complete the proof of this claim, we only need to notice that $3\sqrt{\log g_2}\le e_0$, and hence
\[
\left(3me_0\right)^2\geq 9\cdot m^2\cdot 0\log g_2=81m^2\log g_2
\]
while
\[
25me_1\leq 25m\cdot 3.222m \log g_2=80.55m^2\log g_2.
\]
Now we may turn to study the function
\[
f(S):=2\left( \frac{e_0m}{\sqrt{\log S}} + \frac{e_1m}{\log S} + \frac{e_0m^2}{S \sqrt{\log S}} + \frac{2e_0m^2+e_1m^2}{S \log S} \right).
\]
We have $f(S_2)=1$ and $f(S)<1$ for $S > S_2$. We show that $\gamma \ge S_2$ by showing that whenever $S > \gamma$ we have $f(S)<1$. Suppose $S > \gamma$. Notice first that

\[\log S>\log \gamma\geq (3me_0)^2\geq \left(3\cdot 2\cdot 3\sqrt{\log 2}\right)^2>224,\]
and hence, $S>e^{224}$. Using the estimates
\[
e_0\ge 3\sqrt{\log g_2}\geq 1, \ e_1m\le \frac{1}{25}(3me_0)^2,\ S> \log S >(3me_0)^2>37m,
\]
we get

\[
\frac{e_0m}{\sqrt{\log S}} + \frac{e_1m}{\log S} + \frac{e_0m^2}{S \sqrt{\log S}} + \frac{2e_0m^2+e_1m^2}{S \log S}<\frac{1}{3}+\frac{1}{25}+\frac{1}{37}+\frac{2}{3}\cdot\frac{1}{37}+\frac{1}{25}\cdot \frac{1}{37}\frac{1}{2},
\]
which proves the lemma.
\end{proof}

\subsection{An axiomatic Baker type theorem}

The proof of Theorem \ref{BAKERTYPE} will be based on the following axiomatic approach given in \cite{TM}.

First we need some notations and assumptions.
Let $\Theta_1,...,\Theta_m\in\mathbb{C}^*$ be given. Set  $\overline{l}=(l_1,...,l_m)^T,\quad L=L(\overline{l})=l_1+...+l_m$ When $L\ge L_0$, assume that we have a sequence of simultaneous linear forms
\begin{equation}\label{ASSUMPTION1}
L_{k,j}=L_{k,j}(\overline{l})=B_{k,0}(\overline{l})\Theta_j+B_{k,j}(\overline{l}),
\end{equation}
for $k=0,1,...m, j=1,...,m,$
where $B_{k,j}=B_{k,j}(\overline{l})\in\mathbb{Z}_{\mathbb{I}}$ with $k,j=0,1,...m$ satisfy the determinant condition
\begin{equation*}
\Delta=
\left|\begin{matrix}
& B_{0,0}  &B_{0,1} &   ... & B_{0,m}\\
& B_{1,0}  &B_{1,1} &   ... & B_{1,m}\\
&          & ...    &       &        \\
& B_{m,0}  &B_{m,1} &   ... & B_{m,m}\\
\end{matrix} \right|
\ne 0
\end{equation*}
Furthermore, let $a,b,c,d,b_i,e_i\in\mathbb{R}_{\ge 0}$,  $a,c,c-dm>0$ and suppose that $|B_{k,0}(\overline{l})|\le e^{q(L)}$ and $|L_{k,j}(\overline{l})|\le  e^{-r_j(\overline{l})}$,
where
\begin{alignat*}{1}
&q(L)=aL\log L+b_0L(\log L)^{1/2}+b_1L+b_2\log L+b_3,\\
&-r_j(\overline{l})=(dL-cl_j)\log L+e_0L(\log L)^{1/2}+e_1L+e_2\log L+e_3,
\end{alignat*}
for all $k,j=0,1,...,m$.

Set
\begin{alignat*}{1}
&f=\frac{2}{c-dm},\quad A=b_0+\frac{ae_0m}{c-dm},\quad B=a+b_0+b_1+\frac{ae_1m}{c-dm},\\
&C=am+b_2+\frac{a(dm^2+e_2m)}{c-dm},\quad D=b_0m+\frac{ae_0m^2}{c-dm}\\
&E=(a+b_0+b_1)m+b_2+b_3+\frac{a((2e_0+e_1)m^2+(e_2+e_3)m)}{c-dm}.
\end{alignat*}
and define
\begin{equation*}
\xi(z,H):=
A\left(f\frac{z(f \log H)}{\log H}\right)^{1/2}+
B\frac{ z(f \log H)}{\log H}+
C\frac{\log z(f \log H)}{ \log H}
+D\frac{(\log z(f \log H))^{1/2}}{\log H}.
\end{equation*}
Write
\[
H_0=\max\{ m, L_0,e^{(\gamma\log \gamma)/f},e^{e/f}\}, \quad \gamma=\max\{S_2,1\},
\]
where $S_2$ is the largest solution of the equation
$$S\log S= f(e_0mS(\log S)^{1/2}+ e_1mS+(dm^2+e_2m)\log S     
$$
\begin{equation}\label{ASSUMPTION12}
+e_0 m^2(\log S)^{1/2}+2e_0m^2+e_1m^2+e_2m+e_3m).
\end{equation}

\begin{Theorem}\label{BAKERAXIOM2}
With the assumptions just stated, we have
\begin{equation*}
\left|\beta_0+\beta_1\Theta_1+...+\beta_m\Theta_m\right|\ge 
\frac{1}{2e^{E}}
H^{-\frac{a}{c-dm}-\epsilon(H)}.
\end{equation*}
for all 
$$\overline{\beta}=(\beta_0,\beta_1,...,\beta_m)^T \in 
\mathbb{Z}_{\mathbb{I}}^{m+1}\setminus\{\overline{0}\}$$
and
\begin{equation*}
H=\prod_{i=1}^{m}(2mH_i)\ge H_0,\quad  H_i\ge h_i=\max\{1,|\beta_i|\} 
\end{equation*}
with an error term
\begin{equation*}
\epsilon(H)=\xi(z,H).
\end{equation*}
\end{Theorem}

\begin{Lemma} \label{NESTEDLOG}
\cite{HANCLETAL} The inverse function $z(y)$ of the function
\begin{equation}\label{zzy}
y(z)= z \log z, \quad  z \geq \frac{1}{e},
\end{equation}
is strictly increasing. 
Define $z_0(y)=y$ and $z_n(y)=\frac{y}{\log z_{n-1}}$ for $n\in\mathbb Z^+$. 
Suppose $y>e$, then we have 
\begin{equation}\label{z1z3zz2z0}
z_1<z_3<\cdots <z<\cdots <z_2<z_0. 
\end{equation}
Thus the inverse function may be given by the infinite nested logarithm fraction
\begin{equation}\label{logloglog} 
z(y) =z_{\infty}= \frac{y}{\log \frac{y}{\log \frac{y}{\log ...}}},\quad y>e. 
\end{equation}
In particular,
\begin{equation}\label{loglog} 
z(y)< z_2(y) = \frac{y}{\log \frac{y}{\log y}},\quad y>e. 
\end{equation}
\end{Lemma}

\subsection{Proof of Theorem \ref{BAKERTYPE} and corollaries}

We have $a=c=$ and $d=b_2=e_2=b_3=e_3=0$, and thus we may apply Theorem \ref{BAKERAXIOM2} with
\begin{alignat*}{1}
f & =\frac{2}{c-dm}=2,\quad A=b_0+\frac{ae_0m}{c-dm}=b_0+e_0m,\\
B& =a+b_0+b_1+\frac{ae_1m}{c-dm}=1+b_0+b_1+e_1m,\\
C& =am+b_2+\frac{a(dm^2+e_2m)}{c-dm}=m, \quad D=b_0m+\frac{ae_0m^2}{c-dm}=b_0m+e_0m^2, \\
E& =(a+b_0+b_1)m+b_2+b_3+\frac{a((2e_0+e_1)m^2+(e_2+e_3)m)}{c-dm}=
(1+b_0+b_1)m+(2e_0+e_1)m^2
\end{alignat*}
for the numbers $\Theta_j:=e^{\alpha_j}$.

This completes the proof of Theorem \ref{BAKERTYPE} and Corollary \ref{Corollary2.2}.

\begin{proof}[Proof of Corollary \ref{Corollary2.3}] We start from the estimate (\ref{BAKERLOWER}), writing
\begin{equation}\label{kukka}
2e^{E}H^{1+\epsilon(H)}=
\hat H^{1+\epsilon(H)+\frac{1}{\log\hat H}\left(\log 2+E +m(1+\epsilon(H))\log(2m)\right)}.
\end{equation}
Next we estimate $\epsilon(H)$ by using the bound \ref{WEAKERe} for the values of the function $z$:
\begin{multline*}
\epsilon(H) \le A\frac{2\sqrt{\rho}}{\sqrt{\log\log H}}
+B\frac{2\rho}{\log \log H}
+\frac{C}{\log H}\log\left( \frac{2\rho \log H}{\log\log H}\right)
+\frac{D}{\log H}\sqrt{\log\left( \frac{2\rho \log H}{\log\log H}\right)}\\
\le
A\frac{2\sqrt{\rho}}{\sqrt{\log\log H}}
+B\frac{2\rho}{\log\log H}
+C\frac{\log\log H}{\log H}
+D\frac{\sqrt{\log\log H}}{\log H}
\end{multline*}
\begin{equation}\label{perhonen}
= \frac{1}{\sqrt{\log\log H}}\left(2\sqrt{\rho}A+B\frac{2\rho}{\sqrt{\log\log H}}
+C\frac{(\log\log H)^{3/2}}{\log H}+D\frac{\log\log H}{\log H}\right)
\end{equation}
Case a]. Assume $\log g_1\le \log g_2+\log g_4$, so that $b_1=0$. 
Because $g_4 \le g_2$ we have $b_0 \le 3e_0/7$. Now
\[
\frac{B}{\sqrt{\log\log H}} = \frac{1+b_0+b_1+e_1m}{\sqrt{\log\log H}}
\le \frac{1}{3me_0}+\frac{1}{7m}+\frac{\sqrt{e_1m}}{5},
\]
where we used \eqref{valinta} and \eqref{12me1}.
With the estimates $m \ge 2$, $e_0 \ge 3\sqrt{\log 2}$ and $e_1\le\ 2\log(1+g_3)+2\log(2)+1$ we obtain
\begin{equation}\label{caseaa}
B \frac{2\rho}{\sqrt{\log\log H}} 
\le 0.283+0.633\sqrt{m}+0.580\sqrt{m}\sqrt{\log(1+g_3)}.
\end{equation}
Case b]. Here we assume $\log g_1>\log g_2+\log g_4$. This time
\begin{align*}
\frac{B}{\sqrt{\log\log H}} & = \frac{b_0+e_1(m+1)-2\log(1+g_3)-2\log 2}{\sqrt{\log\log H}}
\le \frac{1}{7m}+\frac{\sqrt{e_1m}}{5}+\frac{\sqrt{e_1}}{5\sqrt m}.
\end{align*}
Estimating $m \ge 2$ and $e_1\le\ \log g_1 +\log(1+g_3)+2\log(2)+1$ gives us
\begin{equation}\label{casebee}
B\frac{2\rho}{\sqrt{\log\log H}} 
\le 0.594+0.633\sqrt m+(0.290+0.410\sqrt{m})\sqrt{\log(g_1(1+g_3))}.
\end{equation}
It's easy to see that the contribution of the terms in parenthesis of \eqref{kukka}, the terms with $C$ and $D$ in \eqref{perhonen} and the constants in \eqref{caseaa} and \eqref{casebee} is small,
less than $1$ for example. So we have shown
\begin{alignat*}{1}
a]& \quad \hat A\le 1+(3.036+7.084m)\sqrt{\log g_2}+0.633\sqrt{m}+0.580\sqrt{m}\sqrt{\log(1+g_3)}, \quad g_1 \le g_2g_4\\
b]& \quad \hat A\le 1+(3.036+7.084m)\sqrt{\log g_2}+0.633\sqrt m \\
& \qquad \qquad \qquad \qquad +(0.290+0.410\sqrt{m})\sqrt{\log(g_1(1+g_3))}, \quad g_1> g_2g_4.
\end{alignat*}
\end{proof}

\begin{proof}[Proof of Corollary \ref{Corollary2.4}]
We have
\[
M\le \hat h:=\prod_{i=1}^{m}h_i\le M^m,\quad h_i=\max\{1,|\beta_i|\},\quad i=0,1,...,m,
\]
where wlog we may suppose that $M=\underset{0\le i\le m}\max\{h_i\}=h_0$. Thus, by Corollary \ref{Corollary2.3} we have
\begin{multline*}
\left|\beta_0 e^{\gamma_0}+\beta_1 e^{\gamma_1}+...+\beta_m e^{\gamma_m}\right|= 
|e^{\gamma_0}|\left|\beta_0 +\beta_1 e^{\gamma_1-\gamma_0}+...+\beta_m e^{\gamma_m-\gamma_0}\right|\\
> \frac{|e^{\gamma_0}|}{\hat h^{1+\hat\epsilon(\hat h)}}
\ge \frac{|e^{\gamma_0}|}{h_1\cdots h_mM^{m\hat\epsilon(\hat h)}}
= \frac{M^{1-\hat \delta(M)}}{h_0h_1\cdots h_m},\quad 
\hat \delta(M):= m\hat\epsilon(\hat h)-\log(|e^{\gamma_0|})/\log M,
\end{multline*}
where
\[
\hat\epsilon(\hat h)=\hat A(\overline{\eta})/\sqrt{\log\log \hat h},\quad 
\overline{\eta}=\overline{\gamma}-\gamma_0\overline{1}=(0,\gamma_1-\gamma_0,...,\gamma_m-\gamma_0).
\]
Hence
\begin{equation*}
\hat B(\overline{\gamma})=1+m\hat A(\overline{\eta}).
\end{equation*}

Note that
\begin{equation}\label{g1g3g4}
g_1(\overline{\eta})\le g_1(\overline{\gamma}),\quad 
g_3(\overline{\eta})\le 2g_3(\overline{\gamma}),\quad
g_4(\overline{\eta})\le 1+g_1(\overline{\gamma}).
\end{equation}
First suppose that $m\ge 3$, and so
\begin{equation*}
g_1(\overline{\gamma})(g_3(\overline{\gamma})+1)\ge 3
\end{equation*}
because the coordinates of $\overline{\gamma}$ are distinct.
By the mean value theorem we can now further estimate
\begin{align}
\sqrt{\log(g_1(\overline{\gamma})(1+2g_3(\overline{\gamma})))}
& \le \sqrt{\log(g_1(\overline{\gamma})(1+g_3(\overline{\gamma})))}
+\frac{\log(\frac{g_1(\overline{\gamma})(1+2g_3(\overline{\gamma}))}{g_1(\overline{\gamma})(1+g_3(\overline{\gamma}))})}{2\sqrt{\log(g_1(\overline{\gamma})(1+g_3(\overline{\gamma})))}} \nonumber  \\
& \le \sqrt{\log(g_1(\overline{\gamma})(1+g_3(\overline{\gamma})))}+ \frac{\log \left(1+\frac{g_3(\overline{\gamma})}{2+g_3(\overline{\gamma})}\right)}{2\sqrt{\log(3)}}\\ 
& \le \sqrt{\log(g_1(\overline{\gamma})(1+g_3(\overline{\gamma})))}+0.331. \label{meanvalue}
\end{align}

Now we are ready to apply the estimates (\ref{ahatA}) and (\ref{bhatA}) of Corollary \ref{Corollary2.3}.\\
Case a].
By (\ref{ahatA}) and (\ref{g1g3g4}) we get
$$
\hat A(\overline{\eta})
\le 1+(3.036+7.084m)\sqrt{\log g_2(\overline{\eta})}+0.633\sqrt{m}
+0.580\sqrt{m}\sqrt{\log(1+g_3(\overline{\eta}))}
$$
$$
\le 1+0.633\sqrt{m}+(3.036+0.580\sqrt{m}+7.084m)\sqrt{\log(g_1(\overline{\gamma})(1+2g_3(\overline{\gamma})))}
$$
\begin{equation}\label{hatA633a}
\le 1+0.633\sqrt{m}
+(3.036+0.580\sqrt{m}+7.084m)
\left(\sqrt{\log(g_1(\overline{\gamma})(1+g_3(\overline{\gamma})))}+0.331\right),
\end{equation}
if $g_1(\overline{\eta})\le g_2(\overline{\eta})g_4(\overline{\eta})$.\\
Case b]. Similarly we get
\begin{equation}\label{hatA633b}
\hat A(\overline{\eta})
\le 1+0.633\sqrt{m}
+(3.326+0.41\sqrt{m}+7.084m)
\left(\sqrt{\log(g_1(\overline{\gamma})(1+g_3(\overline{\gamma})))}+0.331\right),
\end{equation}
if $g_1(\overline{\eta})> g_2(\overline{\eta})g_4(\overline{\eta})$.
We see that the estimate (\ref{hatA633a}) dominates (\ref{hatA633b})
implying
\begin{multline*}
\frac{\hat B(\overline{\gamma})}{m^2\sqrt{\log(g_1(\overline{\gamma})(1+g_3(\overline{\gamma})))}}\\
\le \frac{1+m\left(1+0.633\sqrt{m}
+(3.036+0.580\sqrt{m}+7.084m)
\left(\sqrt{\log(g_1(\overline{\gamma})(1+g_3(\overline{\gamma})))}+0.331\right)\right)}
{m^2\sqrt{\log(g_1(\overline{\gamma})(1+g_3(\overline{\gamma})))}}\\
\le \frac{1}{9\sqrt{\log 3}}
+\frac{1}{3\sqrt{\log 3}}
+\frac{0.633}{\sqrt{3\log 3}}
+\left(\frac{3.036}{3}+\frac{0.580}{\sqrt{3}}+7.084\right)
\left(1+\frac{0.331}{\sqrt{\log 3}}\right) < 12.
\end{multline*}

The case $m=2$ will be divided to three parts:
\begin{enumerate}
\item $g_1(\overline{\gamma})(1+g_3(\overline{\gamma}))=2$, and hence $\gamma_0,\gamma_1,\gamma_2\in\{0,1,-1\}$
\item $g_1(\overline{\gamma})(1+g_3(\overline{\gamma}))=3$, and hence $\gamma_0,\gamma_1,\gamma_2\in\{0,1,-1,2,-2\}$ (excluding the cases above) or
$\gamma_0,\gamma_1,\gamma_2\in\{0,1/2,-1/2\}$
\item $g_1(\overline{\gamma})(1+g_3(\overline{\gamma}))\ge 4$.
\end{enumerate}
The proof is easy to complete with simple calculations in all of this cases. Now we can turn to proof of (\ref{AKTHATM}). This is very straightforward. By  (\ref{g1g3g4}), we have
\begin{multline*}
e_0(\overline{\eta})^2\le 9\log g_2(\overline{\eta})+3.25\log g_4(\overline{\eta})\le 
9\log(2g_1(\overline{\gamma})(1+g_3(\overline{\gamma})))
+3.25\log(g_1(\overline{\gamma})(1+g_3(\overline{\gamma})))\\ 
\le 21.25\log(g_1(\overline{\gamma})(1+g_3(\overline{\gamma}))). 
\end{multline*}
Thus
\begin{equation*} 
\log\hat H_{0,AKT}(\overline{\eta})\le 
95.625m^2\log(g_1(\overline{\gamma})(1+g_3(\overline{\gamma}))) 
e^{191.25m^2\log(g_1(\overline{\gamma})(1+g_3(\overline{\gamma})))}.
\end{equation*}
\end{proof}

\subsection{Proofs of comparisons and examples}\label{tokattodistukset}
\begin{proof}[Proof of Theorem \label{vertailu}]
By further estimating the results of Corollary \ref{Corollary2.3} we get the bound for $\hat A_{AKT}(\overline{\alpha})$. The bound for $\log\hat H_{0,AKT}(\overline{\alpha})$ requires a bit more work. We have
\begin{multline}\label{e0alpha2} 
\log\hat H_{0,AKT}(\overline{\alpha})\le 
\frac{9}{2}m^2e_0(\overline{\alpha})^2e^{9m^2e_0(\overline{\alpha})^2}  \\ \le \frac{9}{2}m^2(9\log g_2(\overline{\alpha})+3.25\log g_4(\overline{\alpha})) 
e^{9m^2(9\log g_2(\overline{\alpha})+3.25\log g_4(\overline{\alpha}))}.
\end{multline}

Now, by (\ref{gggggg}), we obtain
\begin{equation*}
e_0(\overline{\alpha})^2\le 9\log g_2(\overline{\alpha})+3.25\log g_4(\overline{\alpha})
\le 12.25\log(2g_1(\overline{\alpha})\tilde g_3(\overline{\alpha})),
\end{equation*}
and hence
\begin{equation}\label{loghatH0} 
\log\hat H_{0,AKT}(\overline{\alpha})\le 
55.125m^2\log(2g_1(\overline{\alpha})\tilde g_3(\overline{\alpha})) 
e^{110.25m^2\log(2g_1(\overline{\alpha})\tilde g_3(\overline{\alpha}))}.
\end{equation}

Finally we prove that our estimate (\ref{AKTepsilon3}) is better that (\ref{Sepsilon2}). We have
\[
\hat A_{AKT}(\overline{\alpha})< m+13m\sqrt{\log(2g_1\tilde g_3)},
\]
where by the arithmetic-geometric inequality, we get
\[
m\sqrt{\log (2g_1\tilde g_3)}\le m^2+\frac{1}{4}\log (2g_1\tilde g_3).
\]
Thus  we may conclude that
\[
\hat A_{AKT}(\overline{\alpha})< 13m^2+m+\frac{13}{4}\log(2g_2\tilde g_3)
< \hat A_{SA}(\overline{\alpha}).
\]
\end{proof}

\begin{proof}[Proof of Example \ref{Example2.5}]
Now
\begin{alignat*}{3}
g_1&=\textrm{lcm}(y_1,\dots ,y_m)=1,\quad & g_2&=\max_{0\leq j\leq 
m}\{|x_j|+|y_j|\}=1+m,\\
g_3&=\max_{0\leq j\leq m}|\alpha_j|=m,& g_4&=\max_{1\leq j\leq 
m}\left\{1+\frac{|y_j|}{|x_j|}\right\}=2.
\end{alignat*}
Substituting, we obtain
\begin{alignat*}{3}
b_0&=\sqrt{\log(m+1)}+\frac{\log 2}{2\sqrt{\log (m+1)}},\quad & b_1&=0,\\
e_0&=3\sqrt{\log(m+1)}+\frac{\log 2}{2\sqrt{\log (m+1)}}, \quad & e_1&=\log 
2+\log(m+1)+1.
\end{alignat*}
Consequently
\begin{alignat*}{1}
A&=(3m+1)\sqrt{\log(m+1)}+\frac{(m+1)\log 2}{2\sqrt{\log(m+1)}},\\
B&=m\log(m+1)+(1+\log 2)m+\sqrt{\log(m+1)}+1+\frac{\log 2}{2\sqrt{\log(m+1)}},\\
C&=m,\\
D&=(3m^2+m)\sqrt{\log(m+1)}+\frac{(m^2+m)\log 2}{2\sqrt{\log(m+1)}},\\
E&=m^2\log(m+1)+(6m^2+m)\sqrt{\log(m+1)}+(1+\log 
2)m^2+m+\frac{(2m^2+m)\log 2}{2\sqrt{\log(m+1)}}.
\end{alignat*}
By imitating the proof of Corollary \ref{Corollary2.3} we get
\begin{align*}
& \epsilon(H) \log\log H  \le 2\sqrt{\rho}A+B\frac{2\rho}{\sqrt{\log\log H}}
+C\frac{(\log\log H)^{3/2}}{\log H}+D\frac{\log\log H}{\log H} \\
& \le
2\sqrt{\rho}\left((3m+1)\sqrt{\log(m+1)}+\frac{(m+1)\log 2}{2\sqrt{\log(m+1)}}\right) \\
& +
\frac{2\rho}{9m\sqrt{\log(m+1)}}
\left(m\log(m+1)+(1+\log 2)m+\sqrt{\log(m+1)}+1+\frac{\log 2}{2\sqrt{\log(m+1)}}\right) \\
& + \frac{1}{(81m^2\log(m+1))^2}\left(m+(3m^2+m)\sqrt{\log(m+1)}+\frac{(m^2+m)\log 
2}{2\sqrt{\log(m+1)}}\right) \\
&\le 0.962+0.670m+(2.252+6.072m)\sqrt{\log(m+1)}
\end{align*}
implying
\begin{equation*}
\hat A_{AKT}(\overline{\alpha}) 
\le  1+0.670m+(2.252+6.072m)\sqrt{\log(m+1)},
\end{equation*}
since the contribution of the other terms is extremely small.
Moreover, (\ref{e0alpha2}) implies
\begin{equation*}
\log\hat H_{0,AKT}(\overline{\alpha})= 
m^2(40.5\log(m+1)+9.850)e^{m^2(81\log(m+1)+19.699)}.
\end{equation*}
\end{proof}

\begin{proof}[Proof of Example \ref{Example2.7}]
Recall that the number $m+1$ of integral points of absolute value at most $g_3\le r$ satisfies
$$\pi\left(g_3-\frac{1}{\sqrt{2}}\right)^2\le m+1\le \pi\left(g_3+\frac{1}{\sqrt{2}}\right)^2.$$
Thus
\[
g_1=1,\qquad  2 \le g_2 \le \frac{\sqrt{m+1}}{\sqrt{\pi}}+\frac{1}{\sqrt{2}}+1 \le \sqrt{m+1},
\]
\[
g_3 = g_2-1, \qquad  g_4=2,
\]
where we used the fact that $m+1 \ge 9$.
Again, by working similarly as in Example \ref{Example2.5} one gets
\[
\hat A_{AKT}(\overline{\alpha}) \le 1+(1,596+4,294m)\sqrt{\log(m+1)}\]
and
\begin{equation*}
\log\hat H_{0,AKT}(\overline{\alpha})= 
m^2(20,25\log(m+1)+9,604)e^{m^2(40.5\log(m+1)+19,207)}.
\end{equation*}
\end{proof}

\end{document}